\documentclass[11pt]{amsart}

\usepackage[T1]{fontenc}
\usepackage[latin1]{inputenc}
\usepackage[english]{babel}
\usepackage{amsmath,amssymb,amsthm}
\usepackage{enumitem}
\usepackage{mymacros}
\usepackage{hyperref}

\renewcommand{\MR}[1]{}

\newcommand{\tr}{\operatorname{tr}}
\newcommand{\dd}{\partial}
\newcommand{\db}{\bar \partial}

\title[Ideal membership in $H^\infty$: Toeplitz corona approach]
{Ideal membership in $H^\infty$: \\ Toeplitz corona approach}

\author{Michael Hartz}
\address{Michael Hartz, Department of Mathematics, Washington University in St. Louis, One Brookings Drive,
St. Louis, MO 63130, USA}
\email{mphartz@wustl.edu}
\thanks{M.H. was partially supported by a Feodor Lynen Fellowship.}

\author{Brett D. Wick}
\address{Brett D. Wick, Department of Mathematics, Washington University in St. Louis, One Brookings Drive,
St. Louis, MO 63130, USA}
\email{wick@math.wustl.edu}
\thanks{B.D.W. research supported in part by National Science Foundation DMS grant \#1560955.}


\begin{document}

\begin{abstract}
  We study the ideal membership problem in $H^\infty$ on the unit disc. Thus, given functions $f,f_1,\ldots,f_n$
  in $H^\infty$, we seek sufficient conditions on the size of $f$ in order for $f$ to belong to the ideal
  of $H^\infty$ generated by $f_1,\ldots,f_n$. We provide a different proof of a theorem of
  Treil, which gives the sharpest known sufficient condition.
  To this end, we solve a closely related problem in the Hilbert space $H^2$, which is equivalent
  to the ideal membership problem by the Nevanlinna--Pick property of $H^2$.
\end{abstract}

\maketitle

\section{Introduction}

Let $f,f_1,\ldots,f_n$ belong to $H^\infty$, the algebra of all bounded analytic functions
in the open unit disc $\bD$ in $\bC$. The ideal membership problem in $H^\infty$
is the following problem: Find conditions on the size of $f$ which imply that
there exist $g_1,\ldots,g_n \in H^\infty$ such that
\begin{equation*}
  f = \sum_{k=1}^n f_k g_k,
\end{equation*}
that is, $f$ belongs to the ideal of $H^\infty$ generated by $f_1,\ldots,f_n$.
If $f$ is the constant function $1$, this is the corona problem for $H^\infty$. In this
case, Carleson's corona theorem \cite{Carleson62} asserts that $1$ belongs to the ideal generated
by $f_1,\ldots,f_n$ if and only if there exists $C > 0$ such that
\begin{equation*}
  1 \le C \sum_{k=1}^n |f_k(z)| \quad \tfa z \in \bD.
\end{equation*}
For general $f \in H^\infty$, the existence of $C > 0$ such that
\begin{equation}
  \label{eqn:f_naive}
  |f(z)| \le C \sum_{k=1}^n |f_k(z)| \quad \tfa z \in \bD
\end{equation}
is clearly a necessary condition for $f$ to belong to the ideal generated by $f_1,\ldots,f_n$.
But an example of Rao \cite{Rao67} (see also \cite[Chapter VIII, Exercise 3]{Garnett07}) shows
that Condition \eqref{eqn:f_naive} is not sufficient.
This suggests the following problem.

\begin{prob}
  \label{prob:ideal}
  Let $\varphi: [0,\infty) \to [0,\infty)$ be a non-decreasing function.
  For $f,f_1,\ldots,f_n \in H^\infty$, does the condition
  \begin{equation*}
    |f(z)| \le \varphi \left(  \left( \sum_{k=1}^n |f_k(z)|^2 \right)^{1/2} \right) \quad
    \text{ for all } z \in \bD
  \end{equation*}
  imply that $f$ belongs to the ideal of $H^\infty$ generated by $f_1,\ldots,f_n$?
\end{prob}

It is classically known that for functions of the form $\varphi(s) = s^p$, Problem \ref{prob:ideal}
has a negative answer if $p < 2$ and a positive answer if $p > 2$.
In particular, the case $p=3$ of this result shows
that Condition \eqref{eqn:f_naive} implies that $f^3$ belongs to the ideal generated
by $f_1,\ldots,f_n$, which was first proved by Wolff (see \cite[Chapter VIII, Theorem 2.3]{Garnett07}).
A theorem of Treil \cite{Treil02} shows that Condition \eqref{eqn:f_naive} does not imply
that $f^2$ belongs to the ideal generated by $f_1,\ldots,f_n$. Hence, Problem \ref{prob:ideal}
has a negative answer for the function $\varphi(s) = s^2$.

Many authors have studied Problem \ref{prob:ideal} for
functions $\varphi$ which grow slightly more slowly at $0$ than $t \mapsto t^2$, including
Cegrell \cite{Cegrell90,Cegrell94}, Lin \cite{Lin93}, Pau \cite{Pau05} and Trent \cite{Trent05}.
The best known result to date is due to Treil \cite{Treil07}.

\begin{thm}[Treil]
  \label{thm:main}
  Let $\psi$ be a non-increasing non-negative function on $[0,\infty)$
  which satisfies $\int_M^\infty \psi(s) \, ds < \infty$ for some $M > 0$.
  Let $\varphi(s) = s^2 \psi(\log s^{-2})$.
  Let $f,f_1,\ldots,f_n \in H^\infty$
  with
  \begin{equation*}
  |f(z)| \le \varphi \left(  \left( \sum_{k=1}^n |f_k(z)|^2 \right)^{1/2} \right) \quad
  \text{ for all } z \in \bD.
  \end{equation*}
  Then $f$ belongs to the ideal generated by $f_1,\ldots,f_n$.
\end{thm}
In particular, Theorem \ref{thm:main} applies for every $\varepsilon > 0$
to the function $\psi(s) = s^{-1 - \varepsilon}$,
and hence Problem \ref{prob:ideal} has a positive answer for the function
\begin{equation*}
  \varphi(s) = \frac{s^2}{(\log s^{-2})^{1+\varepsilon}}.
\end{equation*}

Treil's proof of Theorem \ref{thm:main} uses a lemma of Nikolski and a version of Nehari's
theorem to reduce Theorem \ref{thm:main} to showing boundedness of a certain bilinear form,
see also the concluding remarks in \cite[Section 4]{Treil07}.

The purpose of this note is to provide a different proof of Treil's theorem.
Our proof uses the Nevanlinna--Pick property of the Hardy space $H^2$ to translate
the ideal membership problem in $H^\infty$ into a Hilbert space problem in $H^2$.
In the case of the corona problem, this translation is known as the Toeplitz corona theorem.
To solve the Hilbert space problem, we adapt the arguments of Treil and the second author of \cite{TW05}.
Specifically, Theorem \ref{thm:main} is reduced to showing boundedness
of a certain linear functional.
The arguments then follow along the lines of Wolff's proof of the corona theorem
\cite{Gamelin80}. As Treil points out in \cite[Section 4]{Treil07}, his proof uses a different approach.

We remark that the idea to translate the ideal membership problem in $H^\infty$
into a Hilbert space problem in $H^2$ already appeared in Trent's paper \cite{Trent05}.
However, Trent's result is somewhat weaker than Theorem \ref{thm:main}. Indeed, he
remarks \cite[p. 574]{Trent05} that with his technique, it is not possible to decrease the exponent
$\alpha$ in $\varphi(s) = \frac{s^2}{(\log s^{-2})^\alpha}$ below $\alpha = \tfrac{3}{2}$.
The main new ingredient of our proof is a better Carleson measure estimate (Lemma \ref{lem:carleson_remaining}),
which allows us to obtain Treil's theorem, in which the critical exponent is $\alpha = 1$.

The remainder of this note is organized as follows. In Section \ref{sec:red_hilbert},
we deduce from a theorem of Leech a generalization of the Toeplitz corona theorem, which
allows us to translate the ideal membership problem into a problem about the Hilbert space $H^2$.
In Section \ref{sec:red_integral}, we reduce our Hilbert space problem to showing
boundedness of a certain integration functional. This is done by adapting the arguments of
\cite{TW05} to the present setting. In Section \ref{sec:carleson_measure},
we collect Carleson measure estimates and embedding results which will
play a key role in establishing boundedness of the integration functional.
Several of these results are due to Treil \cite{Treil07}.
As mentioned above, the main new ingredient is Lemma \ref{lem:carleson_remaining}, which asserts that
a certain measure constructed from the data of the ideal membership problem is Carleson.
Finally, in Section \ref{sec:integral_estimate}, we finish the proof of Theorem \ref{thm:main}
by showing that the integration functional is bounded.

\section{Reduction to a Hilbert space problem}
\label{sec:red_hilbert}

It will be convenient to restate the ideal membership problem using vector-valued functions.
If $\cE,\cF$ are Hilbert spaces, we let $H^\infty(\cE,\cF)$ denote the space
of bounded analytic functions on $\bD$ with values in $B(\cE,\cF)$, the space
of bounded linear operators from $\cE$ to $\cF$. We endow $H^\infty(\cE,\cF)$
with the supremum norm.
In fact, we will only consider $H^\infty(\bC,\bC^n)$ and $H^\infty(\bC^n,\bC)$, whose
elements can naturally be identified with columns and rows of $H^\infty$ functions, respectively.

We will then prove the following more precise version of Theorem \ref{thm:main}.

\begin{thm}
  \label{thm:main_precise}
  Let $\psi: [0,\infty) \to [0,1]$ be a non-increasing function
  which satisfies $\int_0^\infty \psi(s) \, ds \le 1$
  and let $\varphi(s) = s^2 \psi(\log s^{-2})$.
  Let $F \in H^\infty(\bC^n,\bC)$ be of norm at most $1$
  and let $f \in H^\infty$ such that
  \begin{equation*}
    |f(z)| \le \varphi \left(  ||F(z)|| \right) \quad \text{ for all } z \in \bD.
  \end{equation*}
  Then there exists $H \in H^\infty(\bC,\bC^n)$ with
  \begin{equation*}
    f = F H \quad \text{ and } \quad ||H||_\infty \le C,
  \end{equation*}
  where
  $C$ is an absolute constant.
\end{thm}

To see that Theorem \ref{thm:main_precise} implies Theorem \ref{thm:main},
observe that for ideal membership, only the behavior of $\varphi$ near $0$, and hence
only the behavior of $\psi$ near $\infty$ is relevant. In particular, if $\int_M^\infty \psi(s) \, ds < \infty$
and $\psi$ is non-increasing, then by modifying $\psi$ on an interval of finite length and by
multiplying with a positive constant,
we can assume without loss of generality that $\psi \le 1$ and that
$\int_0^\infty \psi(s) \, ds \le 1$.

The algebra $H^\infty$ can be regarded as the multiplier algebra of the Hardy space $H^2$,
which is the reproducing kernel Hilbert space on the unit disc whose reproducing kernel
is the Szeg\H{o} kernel
\begin{equation*}
  \frac{1}{1 - z \ol{w}}.
\end{equation*}
More generally, every $F \in H^\infty(\cE,\cF)$ induces a multiplication operator
$M_F: H^2 \otimes \cE \to H^2 \otimes \cF$ whose norm is the supremum norm of $F$.

The Hardy space $H^2$ is the prototypical example of a so-called complete Nevanlinna--Pick space,
see the book \cite{AM02} for a comprehensive treatment of this circle of ideas.
It is known that the complete Nevanlinna--Pick property of $H^2$ makes it possible to
translate the corona problem into a problem about $H^2$; this result is sometimes referred
to as the Toeplitz corona theorem. It provides an important stepping stone
in the proof of corona theorems for a variety of function algebras, see for example
\cite{CSW11}.
More generally, this translation is possible for the ideal membership
problem, an observation which is also used by Trent \cite[p.\ 580]{Trent05}.

For completeness, we indicate how to obtain such a reduction to a Hilbert space
problem. To this end, we require the following theorem of Leech \cite{Leech14,KR14}, which is
a version of the Douglas lemma for $H^\infty$ (see also \cite[Theorem 8.57]{AM02}).

\begin{thm}[Leech]
  \label{thm:leech}
  Let $\cE_1,\cE_2,\cF$ be Hilbert spaces and let $F_1 \in H^\infty(\cE_1,\cF)$
  and $F_2 \in H^\infty(\cE_2,\cF)$. The following are equivalent:
  \begin{enumerate}[label=\normalfont{(\roman*)}]
    \item There exists $H \in H^\infty(\cE_2,\cE_1)$ of norm at most $1$ such that $F_1 H = F_2$.
    \item The operator inequality $M_{F_1} M_{F_1}^* \ge M_{F_2} M_{F_2}^*$ holds.
  \end{enumerate}
\end{thm}

The following lemma provides the desired reduction to a Hilbert space problem.
The Toeplitz corona theorem is the special case $f=1$.

\begin{lem}
  \label{lem:toeplitz_corona}
  Let $f \in H^\infty$ and $F \in H^\infty(\bC^n,\bC)$ and let $C > 0$.
  The following are equivalent:
  \begin{enumerate}[label=\normalfont{(\roman*)}]
    \item There exists $H \in H^\infty(\bC,\bC^n)$ such that
      \begin{equation*}
        f = F H \quad \text { and } ||H||_\infty \le C.
      \end{equation*}
    \item For every $g \in H^2$, there exists $G \in (H^2)^n$ such that
      \begin{equation*}
        f g = F G \quad \text{ and } \quad ||G||_2 \le C ||g||_2.
      \end{equation*}
  \end{enumerate}
\end{lem}

\begin{proof}
  (i) $\Rightarrow$ (ii) Let $G = H g \in (H^2)^n$. Then
  \begin{equation*}
    F G = F H g = f g
  \end{equation*}
  and
  \begin{equation*}
    ||G||_2 \le ||H||_\infty ||g||_2 \le C ||g||_2.
  \end{equation*}

  (ii) $\Rightarrow$ (i)
  Consider the multiplication operators $M_F: (H^2)^n \to H^2$
  and $M_f: H^2 \to H^2$. We claim that
  \begin{equation*}
    M_F M_F^* \ge \tfrac{1}{C^2} M_f M_f^*.
  \end{equation*}
  To this end, let $g_0 \in H^2$ and let $g \in H^2$ be a unit vector
  such that $||M_f^* g_0|| = \langle M_f^* g_0,g \rangle$. The assumption
  implies that there exists $G \in (H^2)^n$ with $M_f g = M_F G$ and $||G||_2 \le C$.
  Then
  \begin{equation*}
    ||M_f^* g_0|| = | \langle M_f^* g_0,g \rangle| = |\langle g_0, M_F G  \rangle|
    = | \langle M_F^* g_0,G \rangle| \le C ||M_F^* g_0||,
  \end{equation*}
  which proves the claim.

  In this setting, Leech's theorem (Theorem \ref{thm:leech}), applied with
  $\cE_1 = \bC^n,\cE_2 = \cF = \bC$ and $F_1 = F$ and $F_2 = f$,
  implies that there exists $H \in H^\infty(\bC,\bC^n)$ with $||H||_\infty \le 1$
  and $F H = \frac{1}{C} f$, as desired.
\end{proof}

An application of Lemma \ref{lem:toeplitz_corona} shows that in order to prove Theorem \ref{thm:main},
it suffices to prove the following result.

\begin{prop}
  \label{prop:hilbert_space_problem}
  Let $\psi: [0,\infty) \to [0,1]$ be a non-increasing function
  which satisfies $\int_0^\infty \psi(s) \, ds < \infty$
  and let $\varphi(s) = s^2 \psi(\log s^{-2})$.
  Let $F \in H^\infty(\bC^n,\bC)$ be of norm at most $1$
  and let $f \in H^\infty$ such that
  \begin{equation*}
    |f(z)| \le \varphi \left(  ||F(z)|| \right) \quad \text{ for all } z \in \bD.
  \end{equation*}
  Then for all $g \in H^2$, there exists $G \in (H^2)^n$
  such that
  \begin{equation*}
    F G = f g \quad \text{ and } \quad ||G||_2 \le C ||g||_2,
  \end{equation*}
  where $C$
  is an absolute constant.
\end{prop}

\section{Reduction to an integral estimate}
\label{sec:red_integral}

The purpose of this section is to reduce Proposition \ref{prop:hilbert_space_problem} to an
integral estimate. This argument closely follows \cite[Section 1]{TW05}.
We begin with two simple reductions.

\begin{rem}
  \begin{enumerate}[label=\normalfont{(\alph*)},wide]
    \item 
  Let $F = [f_1,\ldots,f_n] \in H^\infty(\bC^n,\bC)$ and $f \in H^\infty$
  be as in Proposition \ref{prop:hilbert_space_problem}.
  As observed by Treil \cite[Remark 0.4]{Treil07},
  it suffices to prove Proposition \ref{prop:hilbert_space_problem} under the additional
  assumption that $f_1,\ldots,f_n$ have no common inner factor, and hence that $F$ has no zeros.
  Indeed, if
  $\theta$ is the greatest common inner divisor of $f_1,\ldots,f_n$, then one readily checks that
  $\widetilde f := f / \theta$ and $\widetilde F := F / \theta$
  satisfy
  \begin{equation*}
    | \widetilde f(z)| \le \varphi ( ||\widetilde F(z)||) \quad
    \text{ for all } z \in \bD.
  \end{equation*}
  Thus, Proposition \ref{prop:hilbert_space_problem} applied to $\widetilde f$ and $\widetilde F$
  yields for $g \in H^2$ a function $G \in (H^2)^n$ with $\widetilde F G = \widetilde f g$
  and hence $F G = f g$.
\item In addition to assuming that $F$ has no zeros, we may also assume
  that the functions $F,f,g$ of Proposition \ref{prop:hilbert_space_problem}
  are all analytic in a neighborhood of $\ol{\bD}$.
  Indeed, this follows from weak compactness of the unit ball of $H^2$
  and a routine argument using dilations of the involved functions.
  \end{enumerate}
\end{rem}

Let now $F \in H^\infty(\bC^n,\bC), f \in H^\infty$ and $g \in H^2$ be functions as in Proposition
\ref{prop:hilbert_space_problem}
which are analytic in an neighborhood of $\ol{\bD}$. Assume further that $F$ has no zeros
in $\ol{\bD}$, so that $||F||^2$ is bounded below on $\ol{\bD}$.
Define a column of $C^\infty$-functions
\begin{equation*}
  \Phi = F^* (F F^*)^{-1} = \|F\|^{-2} F^*
\end{equation*}
and let
\begin{align*}
  \Pi = I - F^* (F F^*)^{-1} F =  I - \|F\|^{-2} F^* F,
\end{align*}
so that $\Pi(z)$ is the orthogonal projection onto $\ker F(z)$.

Observe that $\Phi$ and $\Pi$ are exactly defined as in \cite[Section 1]{TW05}.
We will require the following identities for $\Phi$ and $\Pi$ from \cite{TW05}.

\begin{lem}
  \label{lem:pi_formulae}
  Let $\Phi$ and $\Pi$ be defined as above. Then
  \begin{align*}
    \Pi \db \Phi &= \db \Phi, \\
    \db \Phi &= - (\dd \Pi)^* \Phi, \\
    \dd \db \Phi &= \dd \Pi \db \Phi + (\dd \Pi)^* \Phi F' \Phi
  \end{align*}
  and
  \begin{equation*}
    ||\dd \Pi||^2 = \frac{||F||^2 ||F'||^2 - |F' F^*|^2}{||F||^4} = ||F||^2 || \db \Phi||^2.
  \end{equation*}
 Moreover, $\Pi \partial \Pi = 0$ and $(\dd \Pi)^* \Pi = 0$.
\end{lem}

\begin{proof}
  The first three identities are contained in \cite[Lemma 1.2]{TW05}.
  The same lemma shows that
  \begin{equation}
    \label{eqn:dd_Pi}
    \db \Phi F = - (\dd \Pi)^*,
  \end{equation}
  so that
  \begin{equation*}
  ||\dd \Pi||^2 = || \db \Phi F F^* (\db \Phi)^*||^2 = \|F\|^2 || \db \Phi||^2.
  \end{equation*}
  On the other hand, we also infer from Equation \eqref{eqn:dd_Pi} that
  $\dd \Pi$ has rank $1$,
  so its operator norm coincides with its Hilbert-Schmidt norm.
  According to \cite[Lemma 1.2]{TW05}, we have
  \begin{equation*}
    \dd \Pi = - F^* (F F^*)^{-1} F' \Pi,
  \end{equation*}
  hence
  \begin{align*}
    ||\dd \Pi||^2 &= \tr( \dd \Pi (\dd \Pi)^*)
    = \tr(F^* (F F^*)^{-1} F' \Pi (F')^* (F F^*)^{-1} F) \\
    &= ||F||^{-2} F' \Pi (F')^*
    = \frac{||F||^2 ||F'||^2 - |F' F^*|^2}{||F||^4},
  \end{align*}
  where in the third step, we used the fact that $\tr(A B) = \tr(B A)$ for matrices $A,B$ of
  appropriate sizes.
  The final two identities are contained in \cite[Corollary 1.3]{TW05}.
\end{proof}

Recall that we would like to find $G \in (H^2)^n$ with $F G = f g$.
Note that
\begin{equation*}
  F \Phi f g =  f g,
\end{equation*}
so $G_0 := \Phi f g$ is a non-analytic solution of
this problem.
Moreover, $|f| \le ||F||$ by the assumption of Proposition \ref{prop:hilbert_space_problem},
hence $||G_0||_2 \le ||g||_2$.
We will correct $G_0$ to be analytic. Let $\bT$ denote the unit circle. Then more precisely, we wish to find
$v \in (L^2(\bT))^n$ such that $G_0 -v \in (H^2)^n$ and $v(z) \in \ker F(z)$ for almost every $z \in \bT$,
since then $F (G_0 - v) = F G_0 = f g$.

To this end, we will use Green's formula. In the sequel,
let $A$ denote the planar Lebesgue measure on $\bD$ and define a measure $\mu$ on $\bD$ by
\begin{equation*}
  d \mu = \frac{2}{\pi} \log \frac{1}{|z|} d A.
\end{equation*}
Let $m$ be the normalized Lebesgue measure on $\bT$.
Moreover, we write
\begin{equation*}
  \widetilde \Delta = \frac{1}{4} \Delta = \dd \db.
\end{equation*}

\begin{lem}[Green's formula]
  \label{lem:Green}
  Let $u$ be a function which is $C^2$ in a neighborhood of $\ol{\bD}$. Then
  \begin{equation*}
    \int_\bT u \, dm - u(0) = \int_{\bD} \widetilde \Delta u \, d \mu.
  \end{equation*}
\end{lem}

Suppose now that $g \in H^2$ has norm at most $1$.
The requirement that $G_0 - v \in (H^2)^n$ is equivalent to demanding that
\begin{equation}
  \label{eqn:H^2_perp}
  \int_{\bT} \langle G_0, h \rangle dm = \int_{\bT} \langle v,h \rangle \, dm
\end{equation}
for all $h \in ((H^2)^n)^\bot \subset (L^2)^n$.
Observe that every $h \in ((H^2)^n)^\bot$ extends to be co-analytic
in $\bD$ with $h(0) = 0$, and let $D$ denote the dense subspace
of all elements of $( (H^2)^n)^\bot$ that extend to be co-analytic in a neighborhood
of $\ol{\bD}$. Then, clearly, it suffices to verify Equation \eqref{eqn:H^2_perp}
for all $h \in D$. For such $h$, we may apply Green's formula (observe that $G_0$
is $C^\infty$ in a neighborhood of $\ol{\bD}$ as well) to deduce that
\begin{equation*}
  \int_{\bT} \langle G_0,h \rangle dm = \int_{\bD} \widetilde \Delta [ \langle \Phi f g, h \rangle ] d \mu
  = \int_{\bD} \dd [ \langle \db \Phi f g, h \rangle ] d \mu,
\end{equation*}
where we have used that $f,g$ are analytic and $h$ is co-analytic.
Lemma \ref{lem:pi_formulae} shows that
$\Pi \db \Phi = \db \Phi$, so that
\begin{equation*}
  \int_{\bD} \dd [ \langle \db \Phi f g, h \rangle ] d \mu,
   = \int_{\bD} \dd [ \langle \db \Phi f g, \Pi h \rangle ] d \mu
   = \int_{\bD} \dd [ \langle \db \Phi f g, \xi \rangle ] d \mu,
\end{equation*}
where $\xi = \Pi h$. If we can show that there exists a constant $C_0 > 0$ such that
\begin{equation*}
  \left| 
   \int_{\bD} \dd [ \langle \db \Phi f g, \xi \rangle ] d \mu \right|
   \le C_0 ||\xi||_2 \quad \text{ for all } \xi = \Pi h, h  \in ((H^2)^n)^{\bot},
\end{equation*}
then it follows from the Riesz representation theorem that there exists $v \in (L^2)^n$ with $||v||_2 \le C_0$
such that
\begin{equation*}
  \int_{\bD} \dd [ \langle \db \Phi f g, \xi \rangle ] d \mu = \int_{\bT} \langle v,\xi \rangle \, dm
\end{equation*}
for all $\xi = \Pi h$. By replacing $v$ with $\Pi v$,
we may assume that $v$ belongs to $\ran \Pi = \ker F$.
Since the left-hand side equals $\int_{\bT} \langle G_0,h \rangle \, dm$, this
$v$ will then satisfy
\begin{equation*}
  \int_\bT \langle G_0,h \rangle \, dm = \int_{\bT} \langle v, \Pi h \rangle \, dm
  = \int_{\bT} \langle v, h \rangle \, dm
\end{equation*}
for all $h \in D$, so that $G := G_0 - v \in (H^2)^n$ satisfies $F G = f g$.
Moreover,
\begin{equation*}
  ||G||_2 \le ||G_0||_2 + ||v||_2 \le 1 + C_0.
\end{equation*}

Thus, in order to prove Proposition \ref{prop:hilbert_space_problem}, and hence Theorem \ref{thm:main_precise},
it suffices to prove the following result.

\begin{prop}
  \label{prop:main_estimate}
  Let $\psi: [0,\infty) \to [0,1]$ be a non-increasing function
  which satisfies $\int_0^\infty \psi(s) \, ds < \infty$
  and let $\varphi(s) = s^2 \psi(\log s^{-2})$.
  Let $F \in H^\infty(\bC^n,\bC)$ and $f \in H^\infty$ be analytic in a neighborhood
  of $\ol{\bD}$. Suppose that $F$ has no zeros in $\ol{\bD}$ and
  that $||F||_\infty \le 1$. Assume further that
  \begin{equation*}
    |f(z)| \le \varphi(||F(z)||) \quad \text{ for } z \in \bD.
  \end{equation*}
  Then for all
  $g \in H^2$ with $||g||_2 \le 1$,
  the estimate
  \begin{equation*}
    \left| 
     \int_{\bD} \dd [ \langle \db \Phi f g, \xi \rangle ] d \mu \right|
     \le C_0 ||\xi||_2 \quad \text{ for all } \xi = \Pi h, h \in ((H^2)^n)^{\bot}
  \end{equation*}
  holds, where
  $C_0$ is an absolute constant.
  Here, as before, $\Phi = \|F\|^{-2} F^*$
  and $\Pi = I - \|F\|^{-2} F^* F$.
\end{prop}

\section{Carleson measure estimates}
\label{sec:carleson_measure}

As in \cite{TW05}, the proof of Proposition \ref{prop:main_estimate} requires several Carleson measure estimates.
For the remainder of this note, we assume that $\varphi,\psi$ and $f,F,\Pi$
satisfy the assumptions of Proposition \ref{prop:main_estimate}. We also
assume that $f$ is not identically zero, since Proposition \ref{prop:main_estimate} trivially
holds in this case.

The first Carleson measure estimate is given by the following lemma from \cite{Treil07}.

\begin{lem}[Treil]
  \label{lem:carleson}
  Let $\psi,\Pi,F$ be as in Proposition \ref{prop:main_estimate}.
  Then for all $g \in H^2$, the estimate
  \begin{equation*}
    \int_{\bD} ||\dd \Pi||^2 \psi (\log \|F\|^{-2}) |g|^2 \, d \mu
    \le 2 e ||g||_2^2
  \end{equation*}
  holds.
\end{lem}

\begin{proof}
  According to \cite[Lemma 2.3]{Treil07}, there exists a constant $C < \infty$
  such that for all $g \in H^2$, the estimate
  \begin{equation*}
    \int_{\bD} \|F\|^{-4} ( \|F\|^2 \|F'\|^2 - |F' F^*|^2)
    \psi (\log \|F\|^{-2}) |g|^2 \, d \mu
    \le C \int_{\bT} |g|^2 d m
  \end{equation*}
  holds. By Lemma \ref{lem:pi_formulae},
  \begin{equation*}
    ||\dd \Pi||^2 = \|F\|^{-4} ( \|F\|^2 \|F'\|^2 - |F' F^*|^2).
  \end{equation*}
  Moreover, examination of the proof of \cite[Lemma 2.2]{Treil07} shows that we can
  take
  \begin{align*}
    C &\le e \psi(0) + e \int_1^\infty r \, d( - \psi(r)) = e \psi(0) + e \psi(1) + e \int_1^\infty \psi(r) \, d r \\
    &\le e + e \int_0^\infty \psi(r) \, dr \le 2 e,
  \end{align*}
  where we have used integration by parts for the Riemann-Stieltjes integral and the
  fact that $r \psi(r) \to 0$ as $r \to \infty$ since $\psi$ is integrable and non-increasing.
\end{proof}

The following lemma serves as a replacement for \cite[Lemma 2.3]{TW05} and is also taken from \cite{Treil07}.
As explained there, to interpret
\begin{equation*}
  \int_{\bD} || \db [ \ol{f^{1/2}} \xi] ||^2 \, d \mu
\end{equation*}
if $f$ has zeros, observe that away from the finitely many zeros of $f$, this expression
is independent of the choice of branch of square root.

\begin{lem}[Treil]
  \label{lem:d_pi_non_analytic}
  Let $\varphi,\Pi,f,F$ be as in Proposition \ref{prop:main_estimate}.
  Then for any $\xi$ of the form $\xi = \Pi h$, where $h \in ((H^2)^n)^\bot$, the estimate
  \begin{align*}
    \int_\bD \varphi(\|F\|) \, ||\dd \Pi||^2 ||\xi||^2 \, d \mu
    &\le 2 e ||\xi||_2^2 \text{ and } \\
    \int_{\bD} || \db[ \ol{f^{1/2}} \xi] ||^2 \, d \mu &\le 4 e ||\xi||_2^2
  \end{align*}
  holds.
\end{lem}

\begin{proof}
  This follows from \cite[Corollary 2.8]{Treil07} and its proof. Some care must be taken since our
  $\Pi$ is different from the $\Pi$ in \cite{Treil07}.
  To see that \cite[Corollary 2.8]{Treil07} also applies to our $\Pi$,
  note that the results in \cite[Section 2]{Treil07}
  up to Lemma 2.7 there hold
  for any family of projections $\Pi$ with $\Pi \dd \Pi = 0$, which our $\Pi$
  satisfies by Lemma \ref{lem:pi_formulae}.
  Moreover, the $\Pi$ in \cite[Corollary 2.8]{Treil07}, let us call it
  $\Pi'$, is related to our $\Pi$  via $\Pi = I - (\Pi')^T$, so $||\dd \Pi|| = ||\dd \Pi'||$.
  Thus, \cite[Corollary 2.8]{Treil07} also applies to our $\Pi$.
  
  (Whereas \cite[Corollary 2.8]{Treil07}
  is stated with constants $1$ and $2$ in the embeddings, respectively, the proof
  actually gives worse constants $C$ and $2 C$, respectively. Here $C$ is the constant
  of \cite[Lemma 2.2]{Treil07}, which can be taken to be $2 e$,
  see the proof of Lemma \ref{lem:carleson}.)
\end{proof}

To adapt the route of \cite{TW05}, there is one missing Carleson measure estimate.
It corresponds to the argument at the end of \cite[Section 3]{TW05},
see also the remarks in \cite[Section 4]{Treil07}.

We begin with a straightforward modification of Green's formula, cf.\ \cite[Lemma 2]{Trent05} and
the footnote on page 234 of \cite{Treil07}.

\begin{lem}
  \label{lem:green}
  Let $u$ be a continuous function on $\ol{\bD}$ which extends to be $C^2$ in a neighborhood
  of $\ol{\bD}$, except at possibly finitely many points $w_1,\ldots,w_n \in \ol{\bD}$.
  Suppose that the gradient of $u$ is bounded near each $w_j$. For $\varepsilon > 0$, let
  \begin{equation*}
    \Omega_\varepsilon = \bD \setminus \bigcup_{j=1}^n \ol{D_\varepsilon(w_j)}.
  \end{equation*}
  Then
  \begin{equation*}
    \lim_{\varepsilon \searrow 0} \int_{\Omega_\varepsilon} \widetilde \Delta u \, d \mu
    = \int_{\bT} u \, d m - u(0).
  \end{equation*}
\end{lem}

\begin{proof}
  This is proved just like the usual Green's formula by applying the identity
  \begin{equation*}
    \int_{G_{\varepsilon}} (v \Delta u - u \Delta v) d A = \int_{\partial G_{\varepsilon}} \left (v \frac{\partial u}{\partial n} - u \frac{\partial v}{\partial n} \right) d s
  \end{equation*}
  with $v = \log \frac{1}{|z|}$ and $G_\varepsilon = \Omega_\varepsilon \setminus \ol{D_{\varepsilon}(0)}$.
  The hypothesis that the gradient of $u$ is bounded near each $w_j$
  is used to show that the integrals
  \begin{equation*}
    \int_{\partial D_\varepsilon(w)} \frac{\partial u}{\partial n} v \, d s
  \end{equation*}
  tend to zero as $\varepsilon \to 0$ for $w = w_j$ or $w = 0$.
\end{proof}

To obtain the missing Carleson measure estimate, we will
use the following variant of a result of Uchiyama. It is a modification of \cite[Theorem 2.1]{TW05}. The main
difference to that result is that the function $\alpha$ below is not assumed to be subharmonic.
Recall that $\widetilde \Delta = \frac{1}{4} \Delta = \dd \db$.

\begin{lem}
  \label{lem:carleson_non_subharm}
  Let $\alpha : \ol{\bD} \to [0,1]$ be a continuous function which extends
  to be $C^2$ in a neighborhood of $\ol{\bD}$ except at possibly finitely
  many points $w_1,\ldots,w_n \in \ol{\bD}$. Suppose that the gradient
  of $\alpha$ is bounded near each $w_j$.
  For $\varepsilon > 0$, let
  \begin{equation*}
    \Omega_\varepsilon = \bD \setminus \bigcup_{j=1}^n \ol{D_\varepsilon(w_j)}.
  \end{equation*}
  Then for all $g \in H^2$ that are analytic in a neighborhood of $\ol{\bD}$,
  \begin{equation*}
    \limsup_{\varepsilon \searrow 0 } \int_{\Omega_\varepsilon} e^{\alpha(z)} \widetilde \Delta \alpha(z) |g(z)|^2 \, d \mu(z) \le e ||g||_2^2.
  \end{equation*}
\end{lem}

\begin{proof}
  As in \cite[Theorem 2.1]{TW05}, a computation shows that
  \begin{equation*}
    \widetilde \Delta (e^\alpha |g|^2) = e^\alpha \widetilde \Delta \alpha |g|^2 + e^\alpha |\dd \alpha g + \dd g|^2
    \ge e^\alpha \widetilde \Delta \alpha |g|^2
  \end{equation*}
  at every point where $\alpha$ is $C^2$.
  Moreover, since the gradient of $\alpha$ is bounded near each $w_j$, so is the gradient of $e^\alpha |g|^2$.
  Thus, Green's formula in the form of Lemma \ref{lem:green} shows that
  \begin{align*}
    \limsup_{\varepsilon \searrow 0} \int_{\Omega_\varepsilon} e^{\alpha} \widetilde \Delta \alpha |g|^2 \, d \mu
    &\le \limsup_{\varepsilon \searrow 0} \int_{\Omega_\varepsilon} \widetilde \Delta (e^\alpha |g|^2) \, d \mu \\
    &= \int_\bT e^\alpha |g|^2 \, dm - e^{\alpha(0)} |g(0)|^2 \\
    &\le e \int_{\bT} |g|^2 \, dm = e ||g||^2.
    \qedhere
  \end{align*}
\end{proof}

To apply Lemma \ref{lem:carleson_non_subharm}, we require the following computation.
Observe again that the expression below involving $f^{-1/2}$ is independent
of the choice of branch of square root.

\begin{lem}
  \label{lem:measure_laplace}
  Let $f,F$ be as in Proposition \ref{prop:main_estimate}, assume
  that $f$ is not identically zero, and define
  \begin{equation*}
    \alpha = \frac{|f|}{||F||^2} \quad \text{ and } \quad
    \beta = \log( ||F||^2).
  \end{equation*}
  Then
  \begin{equation*}
  \frac{|f|^2}{||F||^6} \left| (F f^{-1/2})' F^* \right|^2
  = \widetilde \Delta \alpha + \frac{|f|}{||F||^2} \widetilde \Delta \beta
  \end{equation*}
  away from the zeros of $f$.
\end{lem}

\begin{proof}
  We will use the following general formula,
  which is proved by direct computation:
  If $h$ is $C^2$-function and $G$ is a row of analytic functions,
  then
  \begin{equation*}
    \widetilde \Delta h( ||G||^2)
    = h''(||G||^2) | G' G^*|^2 + h'(||G||^2) ||G'||^2.
  \end{equation*}
  Applying this formula with $h(t) = \frac{1}{t}$ and $G = F f^{-1/2}$, we find that
  \begin{align*}
    \widetilde \Delta \alpha &= \widetilde \Delta \frac{1}{||F f^{-1/2}||^2} \\
    &= \frac{2 |(F f^{-1/2})' (F f^{-1/2})^{*}|^2}{||F f^{-1/2}||^6}
    - \frac{ || (F f^{-1/2})'||^2}{||F f^{-1/2}||^4} \\
    &= 2 \frac{|f|^2}{||F||^6} |(F f^{-1/2})' F^*|^2 - \frac{|f|^2}{||F||^4} ||(F f^{-1/2})'||^2 \\
    &= \frac{|f|^2}{||F||^6} |(F f^{-1/2})' F^*|^2 \\
    & \quad - \frac{|f|}{||F||^2} \left( \frac{|f|}{||F||^2} ||(F f^{-1/2})'||^2
    - \frac{|f|}{||F||^4} |(F f^{-1/2})' F^*|^2 \right).
  \end{align*}
  It remains to show that
  \begin{equation*}
    \widetilde \Delta \beta = \frac{|f|}{||F||^2} ||(F f^{-1/2})'||^2
    - \frac{|f|}{||F||^4} |(F f^{-1/2})' F^*|^2.
  \end{equation*}
  One way of seeing this is to observe that since $f$ is analytic,
  $\widetilde \Delta \log( |f|^{-1}) = 0$, hence
  \begin{equation*}
    \widetilde \Delta \beta = \widetilde \Delta \log( ||F||^2) = \widetilde \Delta ( \log (||F f^{-1/2}||^2)).
  \end{equation*}
  Applying the formula at the beginning of the proof with $G = F f^{-1/2}$ and $h(t) = \log(t)$,
  we see that
  \begin{align*}
    \widetilde \Delta ( \log (||F f^{-1/2}||^2)) &=
    - \frac{ | (F f^{-1/2})' (F f^{-1/2})^*|^2}{||F f^{-1/2}||^4} + \frac{||(F f^{-1/2})'||^2}{||F f^{-1/2}||^2}  \\
    &= \frac{|f|}{||F||^2} || (F f^{-1/2})'||^2 - \frac{|f|}{||F||^4} |(F f^{-1/2})' F^*|^2,
  \end{align*}
  which finishes the proof.
\end{proof}

The following lemma contains the missing Carleson measure estimate and is the main new ingredient
of our proof.

\begin{lem}
  \label{lem:carleson_remaining}
  Let $f,F$ be as in Proposition \ref{prop:main_estimate} and assume
  that $f$ is not identically zero. Then for $g \in H^2$, the estimate
\begin{equation*}
  \int_{\bD} \frac{|f|^2}{||F||^6} \left| (F f^{-1/2})' F^* \right|^2 |g|^2 \, d \mu
  \le (2 e^2 + e) ||g||_2^2
\end{equation*}
holds.
\end{lem}

\begin{proof}
  By considering dilations of $g$ and applying Fatou's lemma, we see
  that it suffices to prove the lemma for $g \in H^2$
  which are analytic in a neighborhood of $\ol{\bD}$.
  Let $\alpha$ and $\beta$ be defined as in Lemma \ref{lem:measure_laplace}.
  Let $w_1,\ldots,w_n$ denote the zeros of $f$ inside of $\ol{\bD}$.
  For $\varepsilon > 0$, let
  \begin{equation*}
    \Omega_\varepsilon = \bD \setminus \bigcup_{j=1}^n \ol{D_\varepsilon(w_j)}.
  \end{equation*}
  Let $g \in H^2$ be analytic in a neighborhood of $\ol{\bD}$.
  Using Lemma \ref{lem:measure_laplace}, we see that
  \begin{align*}
  \int_{\Omega_\varepsilon} \frac{|f|^2}{||F||^6} \left| (F f^{-1/2})' F^* \right|^2 |g|^2 \, d \mu
  &\le \int_{\Omega_\varepsilon} e^{\alpha}
  \frac{|f|^2}{||F||^6} \left| (F f^{-1/2})' F^* \right|^2 |g|^2 \, d \mu \\
  & = \int_{\Omega_\varepsilon} e^\alpha (\widetilde \Delta \alpha + \frac{|f|}{||F||^2} \widetilde \Delta \beta ) |g|^2 \, d \mu \\
  &\le \int_{\Omega_\varepsilon} e^\alpha \widetilde \Delta \alpha |g|^2 \, d \mu
  +  \int_{\bD} e^\alpha \frac{|f|}{||F||^2} \widetilde \Delta \beta |g|^2 \, d \mu.
  \end{align*}
  The assumption of Proposition \ref{prop:main_estimate} implies that
  $\alpha = \frac{|f|}{||F||^2}$ satisfies $0 \le \alpha \le 1$.
  Moreover, direct computation shows that
  \begin{equation*}
    \widetilde \Delta \beta = \frac{||F||^2 ||F'||^2 - |F^* F'|^2}{||F||^4} = ||\partial \Pi||^2,
  \end{equation*}
  where the second equality follows from Lemma \ref{lem:pi_formulae}. Thus,
  Lemma \ref{lem:carleson} implies that the second summand admits the estimate
  \begin{align*}
    \int_{\bD} e^\alpha \frac{|f|}{||F||^2} \widetilde \Delta \beta |g|^2 \, d \mu
    \le e \int_{\bD} \psi (\log ||F||^{-2}) \widetilde \Delta \beta |g|^2 \, d \mu \le 2 e^2 ||g||_2^2.
  \end{align*}
  To estimate the first summand, we use Lemma \ref{lem:carleson_non_subharm}. Once we observe
  that the gradient of $\alpha$ is bounded near each $w_j$,
  Lemma \ref{lem:carleson_non_subharm} will imply that
  \begin{equation*}
    \limsup_{\varepsilon \searrow 0}
  \int_{\Omega_\varepsilon} e^\alpha \widetilde \Delta \alpha |g|^2 \, d \mu \le e ||g||_2^2,
  \end{equation*}
  so that an application of the monotone convergence theorem will finish the proof.

  To see that the gradient of $\alpha$ is bounded near each $w_j$, observe that
  \begin{align*}
    |\db \alpha| = |\dd \alpha| &= |\dd ( (F f^{-1/2}) (F f^{-1/2})^*)^{-1} )| \\
    &\le ||F f^{-1/2}||^{-3} \left\| F' f^{-1/2} - \frac{1}{2} F f^{-3/2} f' \right\| \\
    &\le ||F||^{-3} \left( ||F' f|| + \frac{1}{2} ||F f'|| \right),
  \end{align*}
  and recall that $||F||$ does not vanish on $\ol{\bD}$.
\end{proof}

\section{Estimating the integral}
\label{sec:integral_estimate}

In this section, we will prove Proposition \ref{prop:main_estimate}, and hence Theorem \ref{thm:main}. The arguments are again closely related to the ones in \cite{TW05}.

Under the assumptions of Proposition \ref{prop:main_estimate}, we would like to estimate the integral
\begin{equation*}
   \int_{\bD} \dd [ \langle \db \Phi f g, \xi \rangle ] d \mu
   = \int_{\bD} \dd \left[ \langle \db \Phi f^{1/2}, \ol{f^{1/2}} \xi \rangle g \right] d \mu.
\end{equation*}
This integral is the sum of three terms
\begin{align*}
  \textup{I} &= \int_{\bD} \langle \dd [ \db \Phi f^{1/2}], \ol{f^{1/2}} \xi \rangle g \, d \mu, \\
    \textup{II} &= \int_{\bD} \langle \db \Phi f^{1/2}, \db[ \ol{f^{1/2}}  \xi] \rangle g \, d \mu, \\
  \textup{III} &= \int_{\bD} \langle \db \Phi , \xi \rangle f g' \, d \mu.
\end{align*}
Note that the integrands are defined away from the finitely many zeros of $f$,
and do not depend on the choice of branch of square root. All computations
below take place away from the zeros of $f$.

\subsection{Estimating I}

To estimate $\textup{I}$, we note that
\begin{equation*}
  \langle \dd  [ \db \Phi f^{1/2}], \ol{f^{1/2}} \xi \rangle
  = \langle \dd \db \Phi, \xi  \rangle f
  + \tfrac{1}{2} \langle \db \Phi, \xi \rangle f'.
\end{equation*}
For the first summand, we use the identity
\begin{equation*}
  \dd \db \Phi = \dd \Pi \db \Phi + (\dd \Pi)^* \Phi F' \Phi
\end{equation*}
of Lemma \ref{lem:pi_formulae}. Since $(\dd \Pi)^* \xi = (\dd \Pi)^* \Pi \xi = 0$, again by Lemma \ref{lem:pi_formulae}, it follows that
\begin{equation*}
  \langle \dd \db \Phi, \xi \rangle = \langle \Phi F' \Phi, \dd \Pi \xi \rangle
  = ||F||^{-4} \langle F^* F' F^*, (\dd \Pi) \xi \rangle.
\end{equation*}
(cf.\ the computation at the bottom of \cite[p.\ 151]{TW05}).
For the second summand, recall
from Lemma \ref{lem:pi_formulae} that
$\db \Phi = - (\dd \Pi)^* \Phi = - ||F||^{-2} (\dd \Pi)^* F^*$.
Combining these equalities, we conclude that
\begin{align*}
  \langle \dd  [ \db \Phi f^{1/2}], \ol{f^{1/2}} \xi \rangle
  &= ||F||^{-4} \langle F^* F' F^*, (\dd \Pi) \xi \rangle f 
  - \tfrac{1}{2} ||F||^{-2} \langle F^*, (\dd \Pi) \xi \rangle f'\\
  &= ||F||^{-4} \langle F^* F' F^* f - \tfrac{1}{2} F^* F F^* f', (\dd \Pi) \xi \rangle\\
  & = ||F||^{-4} f^{3/2} \langle F^* \left(Ff^{-1/2}\right)'F^* , (\dd \Pi) \xi \rangle.
\end{align*}

Thus, by Cauchy-Schwarz,
\begin{align*}
  |\textup{I}| &\le
  \int_{\bD} ||F||^{-3} |f|^{3/2} | (F f^{-1/2})' F^*| \, ||\dd \Pi|| \, ||\xi|| \, |g| \, d \mu \\
  &\le
  \left( \int_\bD |f| \, || \dd \Pi ||^2 \, || \xi||^2 \, d \mu \right)^{1/2}
  \left( \int_\bD |f|^2 ||F||^{-6} | (F f^{-1/2})' F^*|^2 \, |g|^2 \, d \mu  \right)^{1/2}.
\end{align*}
Since $|f| \le \varphi(||F||)$, the first factor is dominated by $\sqrt{2 e} ||\xi||_2$ by Lemma \ref{lem:d_pi_non_analytic}.
The second factor is bounded by $\sqrt{2 e^2 + e} ||g||_2 \le \sqrt{2 e^2 + e}$ by Lemma \ref{lem:carleson_remaining}.
Thus,
\begin{equation*}
  |\textup{I}| \le e \sqrt{4 e + 2} ||\xi||_2.
\end{equation*}

\subsection{Estimating II}

From Cauchy-Schwarz and the assumption on $f$, we deduce that
\begin{equation*}
  |\textup{II}| \le \left(  \int_\bD ||\db \Phi||^2 \varphi(\|F\|) |g|^2 \, d \mu \right)^{1/2}
  \left( \int_\bD ||\db [\ol{f^{1/2}} \xi] ||^2 \, d \mu \right)^{1/2}.
\end{equation*}
Lemma \ref{lem:d_pi_non_analytic} shows that the second term is bounded by $2 \sqrt{e} ||\xi||_2$.
Moreover, from Lemma \ref{lem:pi_formulae}, we see that
\begin{equation}
  \label{eqn:Phi_Pi}
  ||\dd \Pi||^2 = \|F\|^2 || \db \Phi||^2.
\end{equation}
Thus, the square of the first term is equal to
\begin{equation*}
  \int_\bD ||\dd \Pi||^2 \psi( \log \|F\|^{-2}) |g|^2 \, d \mu \le 2 e ||g||^2_2 \le 2 e
\end{equation*}
by Lemma \ref{lem:carleson}.
Therefore,
\begin{equation*}
  |\textup{II}| \le 2 \sqrt{2} e ||\xi||_2.
\end{equation*}

\subsection{Estimating III}

Once again by Cauchy-Schwarz and the assumption on $f$, we have
\begin{equation*}
  |\textup{III}| \le \left( \int_{\bD} || \db \Phi||^2 ||\xi||^2 \varphi(\|F\|)^2 \, d \mu \right)^{1/2}
  \left( \int_{\bD} |g'|^2 d \mu \right)^{1/2}.
\end{equation*}
Since $g$ is analytic, $\dd \db |g|^2 = |g'|^2$, so Green's formula (Lemma \ref{lem:Green}), combined with Fatou's lemma, shows that the second
term is dominated by $||g||_2 \le 1$.
By Equation \eqref{eqn:Phi_Pi}, the square of the first term is equal to
\begin{align*}
  \int_{\bD} ||\dd \Pi||^2 ||\xi||^2 \|F\|^{-2} \varphi(\|F\|)^2 \, d \mu 
   &= \int_{\bD} ||\dd \Pi||^2 ||\xi||^2 \|F\|^{2} \psi(\|F\|)^2 \, d \mu  \\
  &\le \int_{\bD} ||\dd \Pi||^2 ||\xi||^2 \varphi(\|F\|) \, d \mu \\
  &\le 2 e ||\xi||^2_2,
\end{align*}
where we have used that $0 \le \psi \le 1$ in the first inequality and Lemma \ref{lem:d_pi_non_analytic}
in the second inequality.
Thus,
\begin{equation*}
  |\textup{III}| \le \sqrt{2 e} ||\xi||_2.
\end{equation*}

Combining the three estimates, we see that Proposition \ref{prop:main_estimate} holds.

\bibliographystyle{amsplain}
\bibliography{literature}

\end{document}